\documentclass[leqno]{amsart}
\usepackage{amsmath}
\usepackage{amssymb}
\usepackage{amsthm}
\usepackage{enumerate} 
\usepackage[mathscr]{eucal}
\theoremstyle{plain}
\usepackage{tikz}
\newtheorem{theorem}{Theorem}[section]
\newtheorem{lemma}[theorem]{Lemma}
\newtheorem{prop}[theorem]{Proposition}
\theoremstyle{definition}
\newtheorem{definition}[theorem]{Definition}
\newtheorem{remark}[theorem]{Remark}
\newtheorem{conjecture}[theorem]{Conjecture}
\newtheorem{example}[theorem]{Example}

\theoremstyle{remark}  
\usepackage[latin1]{inputenc}
\usepackage{tikz}
\usetikzlibrary{shapes,arrows}

\begin{document}
	
	\title[On subspaces of $\ell_\infty$ and extreme contraction in $\mathbb{L}(\mathbb{X}, \ell_{\infty}^n) $]{ On subspaces of $\ell_\infty$ and  extreme contraction in $\mathbb{L}(\mathbb{X}, \ell_{\infty}^n)$}
\author[Sohel, Sain and  Paul  ]{Shamim Sohel, Debmalya Sain  and Kallol Paul }

	\newcommand{\acr}{\newline\indent}
		\address[Sohel]{Department of Mathematics\\ Jadavpur University\\ Kolkata 700032\\ West Bengal\\ INDIA}
	\email{shamimsohel11@gmail.com}
	\address[Sain]{Departmento de Analisis Matematico\\ Universidad de Granada\\ SPAIN }
	\email{saindebmalya@gmail.com}
	\address[Paul]{Department of Mathematics\\ Jadavpur University\\ Kolkata 700032\\ West Bengal\\ INDIA}
	\email{kalloldada@gmail.com, kallol.paul@jadavpuruniversity.in}
	
	\thanks{ The first author would like to thank  CSIR, Govt. of India, for the financial support in the form of Junior Research Fellowship under the mentorship of Prof. Kallol Paul.  The research of Dr. Debmalya Sain is sponsored by a Maria Zambrano Grant at the University of Granada.
	} 
	

	\begin{abstract}
		
		We investigate different possiblities of subspaces of the space $\ell_{\infty}$ in terms of whether the subspaces are polyhedral or not. We further study finite-dimensional subspaces of $\ell_{\infty}$ which are of the form $\ell_\infty^n$ form some $ n \geq 2.$ As an application of the results we compute
		 the number of extreme contractions for a class of  the space of bounded linear operators. In particular we find the number of extreme contractions of   $\mathbb{L}(\mathbb{X}, \ell_{\infty}^n),$ where $\mathbb{X}$ is a  finite-dimensional polyhedral space.
	\end{abstract}
	
	\maketitle

	\section{Introduction}

	We use the symbols $ \mathbb{X}, \mathbb{Y}$ to denote Banach spaces. All the spaces considered here are over the real field $\mathbb{R}.$ Let $ B_{\mathbb{X}}= \{ x \in \mathbb{X}: \|x\| \leq 1\} $  and $ S_{\mathbb{X}}= \{ x \in \mathbb{X}: \|x\| = 1\} $ denote  the unit ball and the unit sphere of $\mathbb{X},$ respectively. Let $ \mathbb{L}(\mathbb{X}, \mathbb{Y})$ denote the Banach space of all bounded linear operators from $\mathbb{X}$ to $\mathbb{Y},$ endowed with the usual operator norm. For any set $S \subset \mathbb{X},$  $\mid S \mid $ denote the cardinality of $S.$ 
	 Let $A$ be a  non-empty convex subset of $  \mathbb{X},$ an element $ z \in A$ is said to be an extreme point of $A$, whenever $z = (1-t)x + ty,$ for some $t \in (0, 1)$ and $ x, y \in A$ then $ x = y = z.$ 
	The set of all extreme points of any convex set $A$ is denoted by $Ext(A).$  An operator $T \in \mathbb{L}(\mathbb{X}, \mathbb{Y})$ is said to be an extreme contraction if $T$ is an extreme point of $ B_{\mathbb{L}(\mathbb{X}, \mathbb{Y})}.$  It is trivial to observe that if $T$ is an extreme contraction then $\|T\|=1.$ A finite-dimensional Banach space $\mathbb{X}$ is said to be polyhedral if $Ext(B_{\mathbb{X}})$ is finite. 
	Note that finite-dimensional polyhedral spaces are always considered over the real field.  An infinite dimensional  Banach space is said to be polyhedral if its every finite dimensional subspace is polyhedral. For a detailed exposition on polyhedrons and their properties, we refer to \cite{A, SPBB}. Let us now recall the definitions of face and facets of a finite dimensional Banach space.


	\begin{definition}
		Let $\mathbb{X}$ be a  finite-dimensional  Banach space.	A convex set $Q$ is said to be a face of  $B_{\mathbb{X}}$ if either $Q=B_{\mathbb{X}}$ or if we can write $Q= B_{\mathbb{X}} \cap \delta M,$ where $ M$ is a closed half-space in $\mathbb{X}$ containing $B_{\mathbb{X}}$ and $\delta M$ denotes the boundary of $M.$  
		A maximal face $F$ is said to be a facet of $B_{\mathbb{X}}$ i.e.,  a face $F$ is said to be a facet if for any face $Q$ of $B_{\mathbb{X}}$ such that $ F \subseteq Q$ then $F=Q.$
	\end{definition}

 From the above definition a face of $B_{\mathbb{X}}$ can be written as intersection of some facets of  $B_{\mathbb{X}}.$ It should also be noted that if $v \in Ext(B_{\mathbb{X}})$ then $\{v\}$ is also a facet of $B_{\mathbb{X}}.$ Clearly, $\mathbb{X}$ is polyhedral if and only if the number of facets of $B_{\mathbb{X}}$ is finite.\\
	
	Clearly, $\ell_{\infty}^n$ is a polyhedral Banach space. Note that if $\mathbb{X}$ is a subspace of $\ell_{\infty}^n$ with $dim(\mathbb{X})= m ( 1 < m < n)$ then it is not necessarily of the form $\ell_{\infty}^m.$ For a concrete example of such a subspace see remarks of Theorem \ref{star} later on. We  find conditions for a subspace $\mathbb{X}$ to be of the form $\ell_{\infty}^m $ for some $ m, 1 <m<n.$ Moving from $\ell_{\infty}^n$ to the space $\ell_{\infty},$ we observe that the space $\ell_{\infty}$ is not polyhedral but it has   infinite dimensional subspaces like $c_0,$ which are polyhedral and finite-dimensional subspaces (see Example \ref{cos}) which are not polyhedral.	In this article  we aim to study finite dimensional subspaces of $\ell_{\infty}$ that can be embedded into $\ell_{\infty}^r,$ for some $r \in \mathbb{N}.$ We also characterize $m$-dimensional subspaces $\mathbb{X}$ of $\ell_{\infty}$ such that $\mathbb{X}$ is of the form $\ell_\infty^m$ for some $ m, 1 <m<n.$ As an application of the results obtained during the study of finite dimensional subspaces of $\ell_{\infty}^n (\ell_{\infty}),$ we explore the extreme contraction of the space $\mathbb{L}(\mathbb{X}, \ell_\infty^n),$ where $\mathbb{X}$ is a finite-dimensional polyhedral Banach space. Moreover, we can also determine the number of facets  $B_{\mathbb{L}(\mathbb{X}, \ell_{\infty}^n)}$ which provide us better understanding of the geometry of the space. To do so we need the notion of $*$-Property which was introduced (see \cite{SSGP}) in the  space of diagonal matrices to study best coaproximation. Here we first define $*$-Property and weak $*$-Property in  $\ell_{\infty}$  space.

	\begin{definition}
		Let $\mathcal{A}= \{ \widetilde{a_1}, \widetilde{a_2}, \ldots, \widetilde{a_m}\}$ be a set of linearly independent elements in $\ell_\infty,$ where $\widetilde{a_k}= (a_1^k, a_2^k, \ldots, a_n^k, \ldots),$ for each $1 \leq k \leq m.$ 
		\begin{itemize}
			\item[(i)] We define
			\[
			\mathcal{S}_\mathcal{A}= \{ (a_i^1, a_i^2, \ldots, a_i^m) \in \mathbb{R}^m: i \in \mathbb{N}\}.
			\]
			Whenever the context is clear we use $\mathcal{S}$ instead of $\mathcal{S}_\mathcal{A}.$
			Any element of $\overline{\mathcal{S}}$ is said to be a component of $\mathcal{A},$ where $\overline{\mathcal{S}}$ stands for the closure of $ \mathcal{S}$ with respect to the usual Euclidean norm in $\mathbb{R}^m.$ In particular, $(a_i^1, a_i^2, \ldots, a_i^m) \in \mathcal{S}$ is said to be the $i$-th component of $\mathcal{A}.$
			\item[(ii)]  Any two components $\textbf{a}, \textbf{b} \in \overline{\mathcal{S}}$ is said to be equivalent if $\textbf{a}= \pm \textbf{b}.$ For any $\textbf{a} \in \overline{\mathcal{S}},$ the equivalent set of $\textbf{a},$ denoted by $E_{\textbf{a}},$ is defined as 
			\[
			E_{\textbf{a}}=\{ \textbf{b} \in \overline{\mathcal{S}}: \textbf{b}= \pm \textbf{a}\}.
			\]  
			\item[(iii)] An element $\textbf{a} \in \overline{\mathcal{S}}$ is said to satisfy the $*$-Property if there exists $\widetilde{\beta}=(\beta_1, \beta_2, \ldots, \beta_m)\in  \mathbb{R}^m$ such that 
			\[
			| \sum_{k=1}^{m} \beta_k \widehat{a}_k| >   |\sum_{k=1}^{m} \beta_k \widehat{b}_k|, \quad \text{for all}~ \textbf{b} \in \overline{\mathcal{S}} \setminus E_{\textbf{a}},
			\] 
			where $\textbf{a}=(\widehat{a}_1, \widehat{a}_2, \ldots, \widehat{a}_m)$ and $\textbf{b}= (\widehat{b}_1, \widehat{b}_2, \ldots, \widehat{b}_m).$
				\item[(iv)] An element $\textbf{a} \in \overline{\mathcal{S}}$ is said to satisfy the weak $*$-Property if there exists $\widetilde{\beta}=(\beta_1, \beta_2, \ldots, \beta_m)\in \mathbb{R}^m$ such that 
			\[
			|\sum_{k=1}^{m} \beta_k \widehat{a}_k  	| \geq \sup   \bigg\{|\sum_{k=1}^{m} \beta_k \widehat{b}_k|: \textbf{b} \in \overline{\mathcal{S}}\setminus E_{\textbf{a}} \bigg\},
			\]
				where $\textbf{a}=(\widehat{a}_1, \widehat{a}_2, \ldots, \widehat{a}_m)$ and $\textbf{b}= (\widehat{b}_1,\widehat{b}_2, \ldots, \widehat{b}_m).$ 
		\end{itemize}
	\end{definition}

Clearly, if a component satisfies the $*$-Property then it  satisfies the weak $*$-Property but the converse is not necessarily true. The following proposition provides a class of such examples. 

\begin{prop}\label{compare}
	Let $\widetilde{a_1}= ( a+t_1, a+ t_2,\ldots, a+ t_n, \ldots), \widetilde{a_2}= (b-t_1, b- t_2,  \ldots, b-t_n, \ldots) \in \ell_{\infty},$ where $ a > 0, b > 0 $ and $\{t_n\}$ is a strictly decreasing sequence of real numbers in $(0,1)$ converging to $0.$  Let $\mathcal{A}=\{ \widetilde{a_1}, \widetilde{a_2}\}.$ Then each component of $\mathcal{A}$ satisfies the weak $*$-Property whereas   components satisfying the $*$-Property are $(a+t_1, b-t_1), (a, b).$
\end{prop}

\begin{proof}
	Here $\mathcal{S}_{\mathcal{A}}= \{ (a+t_i, b- t_i), (a,b): i \in \mathbb{N}\}.$ By choosing $\beta_1=1, \beta_2=1,$ it is immediate that each component of $\mathcal{A}$ satisfies the weak $*$-Property. We next show that the components $(a+t_1, b-t_1), (a,b)$ satisfies the $*$-Property. Take $\beta_1= b, \beta_2= -a,$ then $|\beta_1 (a+t_1) + \beta_2 (b-t_1)|= | t_1(a+b) | > | t_i (a+b)| = |\beta_1 (a+t_i) + \beta_2 (b-t_i)|, $ for any $i > 1 $ and $ |\beta_1 a + \beta_2 b| =0.$ So, the components $ ( a+t_1, b-t_1 )$ satisfies the $*$-Property. Similarly, choosing $\beta_1= b- t_1, \, \beta_2= -a -t_1,$ we see that the component $(a, b)$ satisfies the $*$-Property. Next we claim that for any $ i > 1, $ the component $(a+ t_i, b-t_i)$ does not satisfy the $*$-Property. Suppose on the contrary that for some $ i> 1,$ the component $ (a+ t_i, b-t_i)$  satisfies the $*$-Property. Then there exists $\beta_1, \beta_2 \in \mathbb{R}$ such that $ |\beta_1 (a+t_i) + \beta_2 (b-t_i)| > |\beta_1 (a+t_1) + \beta_2 (b-t_1)|$ and $ |\beta_1 (a+t_i) + \beta_2 (b-t_i)|> |\beta_1 a + \beta_2 b|.$ We first assume that  $\beta_1 (a+t_i) + \beta_2 (b-t_i)> 0.$ Then 
	 $|\beta_1 (a+t_i) + \beta_2 (b-t_i)| > |\beta_1 (a+t_1) + \beta_2 (b-t_1)|$ only if $ \beta_2 > \beta_1$ and $ |\beta_1 (a+t_i) + \beta_2 (b-t_i)|> |\beta_1 a + \beta_2 b|$ only if $\beta_1 > \beta_2,$ which leads to a contradiction.  Similarly, if we consider $\beta_1 (a+t_i) + \beta_2 (b-t_i) < 0,$ then we get a  contradiction. Thus the component   $(a+t_i, b-t_i)$ does not satisfy the $*$-Property. This completes the proof.
\end{proof}

 Next, we define the $*$-Property in $\ell_{\infty}^n$  space.

\begin{definition}
  	Let $\mathcal{A}= \{ \widetilde{a_1}, \widetilde{a_2}, \ldots, \widetilde{a_m}\}$ be a set of linearly independent elements in $\ell_\infty^n,$ where $\widetilde{a_k}= (a_1^k, a_2^k, \ldots, a_n^k),$ for each $1 \leq k \leq m.$ Then the  $i$-th component $(a_i^1, a_i^2, \ldots, a_i^m) $ of $\mathcal{A}$  is said to satisfy the $*$-Property if there exists $\beta_1, \beta_2, \ldots, \beta_m \in \mathbb{R}^m$ such that 
  	\[
  	| \sum_{k=1}^{m} \beta_k a_i^k | > \max \bigg\{	| \sum_{k=1}^{m} \beta_k a_j^k |: j \in \{1, 2, \ldots, n\} \setminus E_{i}\bigg\},
  	\]
  where $E_{i} = \{ j \in \{1, 2, \ldots, n\}: (a_j^1, a_j^2, \ldots, a_j^m)= \pm (a_i^1, a_i^2, \ldots, a_i^m) \}.$ 
\end{definition}

Clearly the  $*$-property and the weak $*$-property coincide in $\ell_\infty^n$ space. Before we move into the next section we announce that henceforth whenever we say $\widetilde{a_k}$ is an element of $\ell_{\infty}$ ( or $ \ell_{\infty}^n$) we mean $\widetilde{a_k}= (a_1^k, a_2^k, \ldots, a_n^k, \ldots)$ ( or $\widetilde{a_k}= (a_1^k, a_2^k, \ldots, a_n^k)$).

	\section{Main Results}

	We start this section  with some simple propositions.

	\begin{prop}\label{prop}
		Let $\mathbb{X}$ be a polyhedral  Banach space and let $\mathbb{Y}$ be a proper subspace of $\mathbb{X}.$ Suppose that $F$ is a facet of $B_{\mathbb{X}}$ such that $int (F) \cap \mathbb{Y} \neq \emptyset,$ where $int(F)$ is the interior of $F$ with respect to subspace topology of $F.$ Then $F \cap \mathbb{Y}$ is a facet of $B_{\mathbb{Y}}.$
	\end{prop}
	
	\begin{proof}
		Let $ v \in  int(F) \cap \mathbb{Y}.$ Suppose on the contrary that $F \cap \mathbb{Y}$ is not a facet of $B_{\mathbb{Y}}.$ Then there exists a face $F_1$ of $B_{\mathbb{X}}$ such that $F \cap \mathbb{Y} \subsetneq F_1 \cap \mathbb{Y}.$ So, $ v \in F_1 \cap \mathbb{Y}$ and $ F_1 \cap int(F) \neq \emptyset.$ This contradicts the fact that $F$ is a facet of $B_{\mathbb{X}}.$ \\
	\end{proof}

	\begin{prop}\label{weak}
			Let $\mathbb{W}$ be an $m$-dimensional subspace of $\ell_{\infty}$ and let  $\mathcal{A}= \{ \widetilde{a_1}, \widetilde{a_2}, $ \\$ \ldots, \widetilde{a_m}\}$ be a basis of $\mathbb{W}.$  Let $\mathbb{P}$ be the set of all nonequivalent components satisfying the weak $*$-Property.
			\begin{itemize}
				\item[(i)] $ \mid  \mathbb{P} \mid \, \geq m.$
				\item[(ii)] for any $\widetilde{\beta}=(\beta_1, \beta_2, \ldots, \beta_m) \in \mathbb{R}^m,$ 
				$\| \sum_{k=1}^{m} \beta_k \widetilde{a_k} \|= | \sum_{k=1}^{m} \beta_k \widehat{a}_k|,$ for some $      (\widehat{a}_1, \widehat{a}_2, \ldots, \widehat{a}_m) \in \mathbb{P}.$
			\end{itemize} 
	\end{prop}

\begin{proof}
	(i)  Let $\overline{\mathcal{S}}$  denote the set of all components of $\mathcal{A}.$  If $\mathbb{P}$ is infinite then we have nothing to show. Otherwise   let $\mathbb{P}= \{ \textbf{a}_1, \textbf{a}_2, \ldots, \textbf{a}_r\},$  where $\textbf{a}_k=( \widehat{a}_1^k, \widehat{a}_2^k, \ldots, \widehat{a}_m^k),$ for any $ k, 1 \leq k \leq r.$ Suppose on the contrary that $  \mid \mathbb{P} \mid = \, r < m.$ Let $\mathbb{V}_1= span \{ \textbf{a}_1, \textbf{a}_2, \ldots, \textbf{a}_r\}$ and   let $\mathbb{V}_2= span \{ \textbf{a}: \textbf{a} \in \overline{\mathcal{S}}\}.$ Clearly, $\mathbb{V}_1 \subsetneq \mathbb{R}^m =  \mathbb{V}_2.$  So, $\mathbb{V}_1^ \perp \neq \emptyset,$  where 
	$\mathbb{V}_1^{\perp} =  \{ (\gamma_1, \gamma_2, \ldots, \gamma_m) \in \mathbb{R}^m : \sum_{k=1}^m \gamma_k \widehat{a}_k =0,  \, 1 \leq k \leq r \}. $ Take $(\gamma_1, \gamma_2, \ldots, \gamma_m) \in \mathbb{V}_1^\perp $  so that  $ |\sum_{k=1}^{m} \gamma_k \widehat{a}_k^i|=0,$ for any $ 1 \leq i \leq r.$
Consider the mapping	$f: \overline{\mathcal{S}} \to \mathbb{R}$ defined as $f((\widehat{b_1}, \widehat{b_2}, \ldots, \widehat{b_m}))=  | \sum_{k=1}^{m}  \gamma_k \widehat{b_k}|. $ Clearly $f$ is a continuous mapping on the compact set $\overline{\mathcal{S}}$ and so, $f$ attains its maximum at some element $\textbf{c}=(\widehat{c_1}, \widehat{c_2}, \ldots, \widehat{c_m}) \in \overline{\mathcal{S}}.$ Thus $ | \sum_{k=1}^{m}  \gamma_k \widehat{c_k}|= \sup \{ | \sum_{k=1}^{m} \gamma_k \widehat{b_k}|: (\widehat{b_1}, \widehat{b_2}, \ldots, \widehat{b_m}) \in \overline{\mathcal{S}}\} \neq 0,$ otherwise $\mathbb{V}_1^\perp \neq \emptyset.$ Clearly, $\textbf{c} \notin \mathbb{P}$ and $\textbf{c}$ satisfies the weak $*$-Property. This contradiction completes the proposition.\\ 

(ii) Observe that $\| \sum_{k=1}^{m} \beta_k \widetilde{a_k}\|= \sup \{ |\sum_{k=1}^{m} \beta_k a_i^k|: i \in \mathbb{N} \}.$ If the supremum is attained at some $s \in \mathbb{N}$ then $\| \sum_{k=1}^{m} \beta_k \widetilde{a_k}\|= |\sum_{k=1}^{m} \beta_k a_s^k|,$ where $(a_s^1, a_s^2, \ldots, a_s^m) \in \mathbb{P}.$ Otherwise we can find $(\widehat{a}_1, \widehat{a}_2, \ldots, \widehat{a}_m) \in \overline{\mathcal{S}} \setminus \mathcal{S}$ such that	$\| \sum_{k=1}^{m} \beta_k \widetilde{a_k} \|= \sup \{ |\sum_{k=1}^{m} \beta_k a_i^k|: i \in \mathbb{N} \}= | \sum_{k=1}^{m} \beta_k \widehat{a}_k|.$  Clearly, $(\widehat{a}_1, \widehat{a}_2, \ldots, \widehat{a}_m) \in \mathbb{P}$ and this completes the proof.
\end{proof}

In the following theorem we provide a sufficient condition for the embedding of a finite-dimensional subspace of $\ell_{\infty}$ into $\ell_\infty^n,$ for some $n \in \mathbb{N}.$

	\begin{theorem}\label{isometry}
		Let $\mathbb{W}$ be an $m$-dimensional subspace of $\ell_{\infty}$ and let  $\mathcal{A}= \{ \widetilde{a_1}, \widetilde{a_2}, \ldots, $ \\ $ \widetilde{a_m}\}$ be a basis of $\mathbb{W}.$  Suppose that $\mathbb{P}$ is the set of all nonequivalent components satisfying the weak $*$-Property and  $ \mid \mathbb{P} \mid \, = r.$ Then $\mathbb{W}$ can be embedded into $\ell_{\infty}^r.$
		\end{theorem}

	\begin{proof}
		Let $ \mathbb{P}=\{ \textbf{a}_1, \textbf{a}_2, \ldots, \textbf{a}_r\},$  where $\textbf{a}_k= (\widehat{a}_k^1, \widehat{a}_k^2, \ldots, \widehat{a}_k^m),$ for each  $ k, 1 \leq k \leq r.$
		We define   a map $ f : \mathbb{W} \to \ell_\infty^r$ as follows 
   \[ 	f \bigg(\sum_{k=1}^{m} \beta_k \widetilde{a_k}\bigg) =\bigg( \sum_{k=1}^{m} \beta_k \widehat{a_1}^k, \sum_{k=1}^{m} \beta_k \widehat{a_2}^k, \ldots, \sum_{k=1}^{m} \beta_k \widehat{a_r}^k\bigg).\]
   Clearly, $f$ is linear.
		From Proposition \ref{weak},  it is easy to observe that for any $\widetilde{\beta}=(\beta_1, \beta_2, \ldots, \beta_m) \in \mathbb{R}^m,$
		$ \| \sum_{k=1}^{m} \beta_k \widetilde{a_k} \| = 	 | \sum_{k=1}^{m} \beta_k \widehat{a_{t}}^k|,$ for some $ t \in \{ 1, 2, \ldots, r\}.$ Therefore, for any $\beta_1, \beta_2, \ldots, \beta_m \in \mathbb{R},$ 
		\begin{eqnarray*}
			\|\sum_{k=1}^{m} \beta_k \widetilde{a_k}\|&=& \max\bigg\{   \bigg| \sum_{k=1}^{m} \beta_k \widehat{a_1}^k \bigg|,  \bigg| \sum_{k=1}^{m} \beta_k \widehat{a_2}^k \bigg|, \ldots, \bigg| \sum_{k=1}^{m} \beta_k \widehat{a_r}^k \bigg| \bigg\}\\
			&=& \bigg\| \bigg( \sum_{k=1}^{m} \beta_k \widehat{a_1}^k, \sum_{k=1}^{m} \beta_k \widehat{a_2}^k, \ldots, \sum_{k=1}^{m} \beta_k \widehat{a_r}^k\bigg)\bigg\|\\
			&=& \|f(\sum_{k=1}^{m} \beta_k\widetilde{a_k})\|.
		\end{eqnarray*}
		Therefore, $f$ is an isometry. Thus $ \mathbb{W}$ is isometrically isomorphic to $\mathbb{V}= span \{ \widetilde{b_1}, \widetilde{b_2}, $ $ \ldots, \widetilde{b_m}\},$ where $ \widetilde{b_k}= ( \widehat{a_1}^k, \widehat{a_2}^k, \ldots, \widehat{a_r}^k)\in \ell_\infty^{r},$ for any $ 1 \leq k \leq m.$
		 Therefore, $\mathbb{W}$ can be embedded into $\ell_\infty^{r}.$ 
			\end{proof}

Since any subspace of $\ell_{\infty}^n$ is polyhedral, the above theorem also provides a sufficient condition for a finite-dimensional subspace of $\ell_\infty$ to be a polyhedral subspace.

\begin{remark} 
	Let $\mathcal{A}= \{ \widetilde{a}_1, \widetilde{a}_2, \ldots, $ $ \widetilde{a}_m\}$ and $\mathcal{B}= \{ \widetilde{b}_1, \widetilde{b}_2, \ldots, $ $ \widetilde{b}_m\}$ be    bases of two  $m$-dimensional subspaces $\mathbb{V}$ and  $\mathbb{W}$ of $\ell_{\infty}$  respectively.  Suppose that the set of nonequivalent components of $\mathcal{A}$ and $\mathcal{B}$ satisfying the weak $*$-Property  are finite and same. Then using similar arguments given in Theorem \ref{isometry}, it can be seen easily that $\mathbb{V} $ is isometrically isomorphic to $ \mathbb{W}.$
\end{remark}
	
	In Theorem \ref{isometry},  the condition is not optimal, i.e., it is possible that $ \mid \mathbb{P} \mid =r$ but the space can be embedded into $\ell_{\infty}^s,$ for some $s < r,$ which is illustrated by the following example.
	
	\begin{example}\label{example:optimal}
		Let $v_1= (3, \frac{5}{2}, 2, 0, 0, \ldots), ~v_2= (2, \frac{5}{2}, 3, 0, 0, \ldots) \in \ell_{\infty}$ and $\mathbb{Y}= span \{ v_1, v_2 \}.$ Then it is immediate that the set of all nonequivalent  components satisfying the weak $*$-Property is  $\{(3, 2), (\frac{5}{2}, \frac{5}{2}). (2, 3)\}$ and the set of all nonequivalent components satisfying the $*$-Property is $\{ (3,2), (2,3)\}.$ It is easy to verify  that for any $(\beta_1, \beta_2) \in \mathbb{R}^2,$ $\| \beta_1 v_1+ \beta_2 v_2\| = | 3 \beta_1+ 2 \beta_2|$ or  $\| \beta_1 v_1+ \beta_2 v_2\| = | 2 \beta_1+ 3 \beta_2|.$ So, $\mathbb{Y}$ can be embedded into $ \ell_{\infty}^2,$ moreover $\mathbb{Y}$ is isometrically isomorphic to $\ell_{\infty}^2.$
	\end{example}

		The following theorem provides  an  optimal condition  for the embedding of a finite-dimensional subspace of $\ell_{\infty}$ into $\ell_{\infty}^n.$ 
	
	\begin{theorem}\label{number:facet}
			Let $\mathbb{W}$ be an $m$-dimensional subspace of $\ell_{\infty}$ and let  $\mathcal{A}= \{ \widetilde{a_1}, \widetilde{a_2}, \ldots,  $ $ \widetilde{a_m}\}$ be a basis of $\mathbb{W}.$
		Suppose that $\mathbb{P}$ is the set of all nonequivalent components satisfying the weak $*$-Property and $\mathbb{Q}$ is the set of all nonequivalent components satisfying the $*$-Property.
		Let $|\mathbb{P}| = | \mathbb{Q}|=r.$
	  Then
		\begin{itemize}
			\item[(i)]  the number of facets of $B_{\mathbb{W}}$ is $2r,$
			\item[(ii)] $\mathbb{W} $ can be embedded into $\ell_\infty^{n}$
			if and only if $r \leq n.$
			\item[(iii)]  $\mathbb{W} $ is isometrically isomorphic to $\ell_\infty^{m}$ if and only if $r   = m.$\\
		\end{itemize} 
	\end{theorem}
	
	\begin{proof}
		
			(i)  	Let $ \mathbb{P}=\{ \textbf{a}_1, \textbf{a}_2, \ldots, \textbf{a}_r\},$  where $\textbf{a}_k= (\widehat{a}_k^1, \widehat{a}_k^2, \ldots, \widehat{a}_k^m),$ for each  $ k, 1 \leq k \leq r.$ Clearly $\mathbb{Q} \subset \mathbb{P}$ and $|\mathbb{P}|=|\mathbb{Q}|=r$ implies that $\mathbb{P}= \mathbb{Q}.$
				Using Theorem \ref{isometry}, we observe that
				 $ \mathbb{W}$ is isometrically isomorphic to a subspace $\mathbb{V}$ of $\ell_{\infty}^r$ and so, $\mathbb{W}$ is polyhedral. Let $ \mathcal{B}= \{ \widetilde{b_1}, \widetilde{b_2}, \ldots, \widetilde{b_m}\}$ be a basis of $\mathbb{V},$  where $ \widetilde{b_k}= ( \widehat{a_1}^k, \widehat{a_2}^k, \ldots, \widehat{a_r}^k)\in \ell_\infty^{r},$ for each $k,  1 \leq k \leq m.$ Then   $\mathbb{V}= span~ \mathcal{B}.$
				 Moreover, we observe that each component of $\mathcal{B}$ are nonequivalent and  satisfies the $*$-Property. 
		 Let $\pm F_1, \pm F_2, \ldots, \pm F_r$ be the facets of $B_{\ell_\infty^r}.$ Suppose that  for any $\widetilde{x}=(x_1, x_2, \ldots, x_n) \in F_i,$ we have $x_i=1.$  We want to show that $ F_i \cap \mathbb{V}$ is a facet of $B_{\mathbb{V}}.$  Since the $i$-th component of $\mathcal{B}$ satisfies the $*$-Property, there exist $ \beta_1, \beta_2, \ldots, \beta_m \in \mathbb{R}$ such that   
			\[
			\bigg|   \sum_{k=1}^{m} \beta_k \widehat{a}_i^k\bigg| > \max \bigg\{ \bigg| \sum_{k=1}^{m} \beta_k \widehat{a}_j^k\bigg| : j \in \{1,2, \ldots, r\} \setminus \{i\}\bigg\}.
			\]
				For each $k, 1 \leq k \leq m,$ choose 
			 $\gamma_k= \frac{\beta_k}{ \sum_{j=1}^{m} \beta_j \widehat{a}_i^j },$ for any $1 \leq k \leq m.$ Then it is easy to observe that $ \sum_{k=1}^{m} \gamma_k \widehat{a}_i^k=1$ and $ | \sum_{k=1}^{m} \gamma_k \widehat{a}_j^k |< 1,$ for any $ j \in \{ 1, 2,\ldots r\}\setminus \{i\}.$ Therefore, $ \sum_{k=1}^{m} \gamma_k \widetilde{b_k} \in int (F_i) \cap \mathbb{V},$ where $int(F_i)$ is the interior of $F_i$ with respect to the subspace topology of $F_i.$ From Proposition \ref{prop}, it follows that $ (F_i \cap \mathbb{V})$ is a facet of $B_\mathbb{V}.$ Thus for each $j, 1\leq j \leq r,$ $ (F_j \cap \mathbb{V})$ is a facet of $B_{\mathbb{V}}.$ It is easy to verify  that the number of facets of $B_{\mathbb{V}}$ is less than equal to the number of facets of $B_{\ell_{\infty}^r}.$ Therefore, the number of facets of $B_{\mathbb{V}}$ is $2r$ and consequently, the number of facets of $B_{\mathbb{W}}$ is $2r.$ This proves the first part of the theorem.
			 
				(ii) 	We just need to prove the necessary part of the theorem as the sufficient part follows directly from Theorem \ref{isometry}.  From (i) we obtain that the number of facets of $\mathbb{W}$ is $2r.$ If $\mathbb{W}$ can be embedded into $\ell_\infty^n$ then $ 2r\leq 2n$  and so, $ r \leq n.$ 
				
				(iii)  The proof follows from (i) and (ii).
	\end{proof}

	We now present   examples which  illustrate  Theorem \ref{number:facet}.

	\begin{example}\label{example:example1}
		
	(i) 	Let $p, q , a , b \in \mathbb{R}$ such that $p < 0< 1\leq a< b < q.$ Let $\widetilde{a_1}= ( p, a-t_1, a- t_2,\ldots, a- t_n, \ldots), \widetilde{a_2}= (q, b-t_1, b- t_2,  \ldots, b-t_n, \ldots) \in \ell_{\infty},$ where  $0 < t_n < 1,$ for any $n \geq 1$ and $t_n \rightarrow 0.$  Let $\mathbb{Y}_1=span\{ \widetilde{a_1}, \widetilde{a_2}\}.$ Then the components satisfying the weak $*$-Property and the $*$-Property are same, which are $(p, q), (a,b).$ So, from Theorem \ref{number:facet},
		$\mathbb{Y}_1$ is isometrically isomorphic to $\ell_{\infty}^2.$ 
		
		(ii)	Let $v_1= (7, -5, 0, \frac{1}{2}, \frac{2}{3}, \frac{3}{4}, \ldots, 1-\frac{1}{n}, \ldots), v_2=  (-5, 6, 1, \frac{3}{2}, \frac{5}{3}, \frac{7}{4}, $ $ \ldots,  2-\frac{1}{n}, \ldots)  \in \ell_{\infty}$ and let $\mathbb{Y}_2= span\{v_1, v_2\}.$ Observe that the set of nonequivalent components satisfying the weak $*$-Property and the set of nonequivalent components satisfying the $*$-Property are same, which is $\{ (7, -5), (-5, 6), $ $ (1,2)\}.$ So,  from Theorem \ref{number:facet}, $\mathbb{Y}_2$ can be embedded into $\ell_{\infty}^3$ and the number of facets of $B_{\mathbb{Y}_2}= 6.$ Since $\dim \mathbb{Y}_2$ is $ 2,$ $Ext(B_{\mathbb{Y}_2})=6.$
		
		(iii)	Let $v_1= ( 3, 0, 1, \frac{5}{2}, 0, 0, \ldots	), v_2= (2, 5, 0, 4, 0, 0, \ldots),  v_3 =( 1, 1, 5, \frac{7}{2}, 0, 0, \ldots) \in \ell_{\infty}$ and let $\mathbb{Y}_3= span \{v_1, v_2, v_3\}.$ The nonequivalent components satisfying the weak $*$-Property and the  $*$-Property are same, which are $(3, 2, 1), (0, 5, 1),  (1, 0, 5), (\frac{5}{2}, 4 , \frac{7}{2}).$ Therefore, from Theorem \ref{isometry}, $\mathbb{Y}_3$ can be embedded into $\ell_{\infty}^3$ and from Theorem \ref{number:facet} the number of facets of $B_{\mathbb{Y}_3}$ is $8$ and $\mathbb{Y}_3$ is not isometrically isomorphic to $\ell_{\infty}^3.$ Moreover, we observe that any two dimensional subspace of $\mathbb{Y}_3$ has exactly $3$ components satisfying the $*$-Property and therefore, any two dimensional subspace of $\mathbb{Y}_3$ is not isometrically isometric to $\ell_{\infty}^2.$ 
	\end{example}

	Considering the finite-dimensional space $\ell_{\infty}^n,$ we can state the following theorem, the proof of which is in the same spirit as Theorem \ref{number:facet}.

	\begin{theorem}\label{star}
			Let $\mathbb{W}$ be an $m$-dimensional subspace of $\ell_{\infty}^n$ and let  $\mathcal{A}= \{ \widetilde{a_1}, \widetilde{a_2}, \ldots, $ $ \widetilde{a_m}\}$ be a basis of $\mathbb{W}.$	 Suppose that the number of  nonequivalent components of $\mathcal{A}$ satisfying the $*$-Property is $r.$ Then 
		\begin{itemize}
			\item[(i)] the number of facets of $B_{\mathbb{W}}$ are $2r.$
				\item[(ii)] $\mathbb{W}$ can be embedded into $\ell_\infty^s$ if and only if $r \leq s.$
				\item[(iii)] $\mathbb{W}$ is isometrically isomorphic to $\ell_\infty^{m}$ if and only if $r=m.$
		\end{itemize} 
	\end{theorem}  

Considering the two-dimensional subspace $\mathbb{W}_1= span \{ v_1, v_2\}$ of $\ell_{\infty}^3$   where $v_1= (3 ,0, 2), v_2=(0, 3, 2) ,$ it is easy to  observe that each component satisfies the $*$-Property. From Theorem \ref{star}, we conclude that $\mathbb{W}_1$ is not isometrically isomorphic to $\ell_{\infty}^2$ and the number of facets of $B_{\mathbb{W}_1}$ is $6.$ Again, consider a two dimensional subspace $\mathbb{W}_2= span \{ w_1, w_2\}$ of $\ell_{\infty}^3,$ where $w_1= (4, \frac{7}{2}, 3), w_2= (3, \frac{7}{2}, 4).$ It is easy to observe that the components satisfying the $*$-Property are $(4,3), (3,4).$ From Theorem \ref{star}, we conclude that $\mathbb{W}_2$ is isometrically isomorphic to $\ell_{\infty}^2.$\\

We next provide a necessary condition for a finite-dimensional subspace of $\ell_{\infty}$ to be polyhedral.
	
	\begin{theorem}\label{polyhedral}
			Let $\mathbb{Y}$ be an $m$-dimensional subspace of $\ell_{\infty}$ and let  $\mathcal{A}= \{ \widetilde{a_1}, \widetilde{a_2}, \ldots,  $ $ \widetilde{a_m}\}$ be a basis of $\mathbb{Y}.$ If
	 $\mathbb{Y}$ is polyhedral  then  the number of nonequivalent components of $\mathcal{A}$ satisfying the $*$-Property is finite.
	\end{theorem}
	
	\begin{proof}
			Let $\mathbb{Q}$ be the set of all nonequivalent components satisfying the $*$-Property.  
		Let $\mathbb{P}= \{ \textbf{a}_1, \textbf{a}_2, \ldots\}$ be the set of all nonequivalent components satisfying the weak $*$-Property,  where $\textbf{a}_k =(\widehat{a}_k^1, \widehat{a}_k^2, \ldots, \widehat{a}_k^m),$ $ k \in \mathbb{N}.$ 
		Consider $\mathcal{B}= \{ \widetilde{w}_1, \widetilde{w}_2, \ldots, \widetilde{w_m}\},$ where $\widetilde{w_i}= (\widehat{a}_1^i, \widehat{a}_2^i, \ldots) \in \ell_\infty,$ for each $i, 1 \leq i \leq m.$ So, the $k$-th component of $\mathcal{B}$ is $\textbf{a}_k.$ Let $\mathbb{W}= span ~ \mathcal{B}.$
		Now we define a map $f: \mathbb{Y} \to \mathbb{W} $ as $ f( \sum_{k=1}^{m} \beta_k \widetilde{a_k})= \sum_{k=1}^{m} \beta_k \widetilde{w_k}.$ Using Proposition \ref{weak} and following similarly as in  Theorem \ref{isometry}
	 we  observe that $\mathbb{Y}$ is isometrically isomorphic to $\mathbb{W}.$ 
	  Since $\mathbb{Y}$ is polyhedral, $\mathbb{W}$ is polyhedral too. Suppose on the contrary  $\mathbb{Q}$ is infinite and we assume that  $\mathbb{Q}= \{ \textbf{a}_i : i \in \mathbb{I} \subset \mathbb{N}\}.$
	 Observe that any two components of $\mathcal{B}$  are  nonequivalent and  any element of $\mathbb{Q}$  satisfies the $*$-Property.
		  Suppose that $\textbf{a}_k \in \mathbb{Q}$ then    there exists $\widetilde{\beta}=(\beta_1, \beta_2, \ldots, \beta_m) \in \mathbb{R}^m$ such that
		$
		|   \sum_{i=1}^{m} \beta_i \widehat{a}_k^i| > 	|  \sum_{i=1}^{m} \beta_i \widehat{a}_j^i   |,
		$ for any $ j \in \mathbb{N} \setminus\{k\}.$ 	Take 
		\[F_k= \bigg\{(t_1, t_2, \ldots, t_n, \ldots )\in \ell_\infty: t_k=1, ~ |t_j| \leq 1 ~  \forall j \in \mathbb{N} \setminus \{k\}\bigg\}.\]
		It is easy to observe that for any $k \in \mathbb{N},$ $F_k$ is a face of $B_{\ell_{\infty}}.$ 	For each $i, 1 \leq i \leq m,$ choose $\gamma_i = \frac{\beta_i}{\sum_{j=1}^{m} \beta_j \widehat{a}_k^j}.$ Then $ \sum_{i=1}^{m} \gamma_i \widehat{a}_k^i=1$ and $ | \sum_{i=1}^{m} \gamma_i \widehat{a}_j^i | < 1,$ for any $ j \in \mathbb{N} \setminus \{k\}.$
			So, $\sum_{i=1}^{m} \gamma_i \widetilde{w_i} \in F_k \cap \mathbb{W}.$
		Clearly, $F_k \cap \mathbb{W}$ is a face of $B_{\mathbb{W}}$ and for any two $k_1, k_2 \in \mathbb{I},$ $F_{k_1} \cap \mathbb{W} \neq F_{k_2} \cap \mathbb{W}.$ So, for any $k \in \mathbb{I},$ $F_k \cap \mathbb{W}$ is a face of $B_{\mathbb{W}}.$	Since $\mathbb{I}$ is infinite, there exists infinitely many faces in $B_{\mathbb{W}}.$ This implies the number of facets of $B_{\mathbb{W}}$ is infinite, which contradicts that $\mathbb{W}$ is polyhedral. Hence the theorem.
			\end{proof}

	In the following example we exhibit a finite dimensional non-polyhedral  subspace of $\ell_{\infty}$ for which  the number of nonequivalent components satisfying $*$-Property is infinite.

	\begin{example}\label{cos}
	 Consider $\mathcal{A}=\{ \widetilde{a_1}, \widetilde{a_2}\}, $ where $\widetilde{a_1}= (\cos \frac{\pi}{2}, \cos \frac{\pi}{4}, \ldots, \cos \frac{\pi}{2n}, \ldots)$ and $ \widetilde{a_2}= (\sin \frac{\pi}{2}, \sin \frac{\pi}{4}, \ldots, \sin \frac{\pi}{2n}, \ldots).$ Let $\mathbb{Y}_4= span ~\mathcal{A}.$
For each $i \in \mathbb{N},$	the $i$-th component of $\mathcal{A}$ is $(\cos \frac{\pi}{2i}, \sin \frac{\pi}{2i}).$	Clearly, any two components  of $\mathcal{A}$ are nonequivalent. 
		Here $\overline{\mathcal{S}}= \{ (cos \frac{\pi}{2i}, \sin \frac{\pi}{2i}), (1,0): i \in \mathbb{N}\}. $
	Choosing  $\beta_1= \cos \frac{\pi}{2i},  \beta_2=\sin \frac{\pi}{2i},$ we observe that $ |\beta_1 \cos \frac{\pi}{2i}+ \beta_2 \sin \frac{\pi}{2i} |=1 > |\beta_1 \cos \frac{\pi}{2j}+ \beta_2 \sin \frac{\pi}{2j} |,$ for any $ j \in \mathbb{N}\setminus \{i\},$ so that the component $ (cos \frac{\pi}{2i}, \sin \frac{\pi}{2i})$ satisfies the $*$-Property, for each $i \in \mathbb{N}.$ Again taking $\beta_1=1, \beta_2=0$ it is easy to verify that  the component $(1, 0)$ satisfies the $*$-Property.
	So, each component  satisfies the $*$-Property. Therefore, from Theorem \ref{polyhedral}, $\mathbb{Y}_4$ is not polyhedral. So, $\mathbb{Y}_4$ can not be embedded into $\ell_\infty^n,$ for any $n \in \mathbb{N}.$
		\end{example}

The converse of Theorem \ref{polyhedral} is not known to us. We put it in the form of following conjecture.

\begin{conjecture}
	Let $\mathbb{Y}$ be an $m$-dimensional subspace of $\ell_{\infty}$ and let  $\mathcal{A}= \{ \widetilde{a_1}, \widetilde{a_2}, \ldots,  $ $ \widetilde{a_m}\}$ be a basis of $\mathbb{Y}.$ Then
	$\mathbb{Y}$ is polyhedral  if and only  if  the number of nonequivalent components of $\mathcal{A}$ satisfying the $*$-Property is finite.
\end{conjecture}  
We end this section by presenting a scheme  of different kind of possible subspaces of $\ell_{\infty}$.

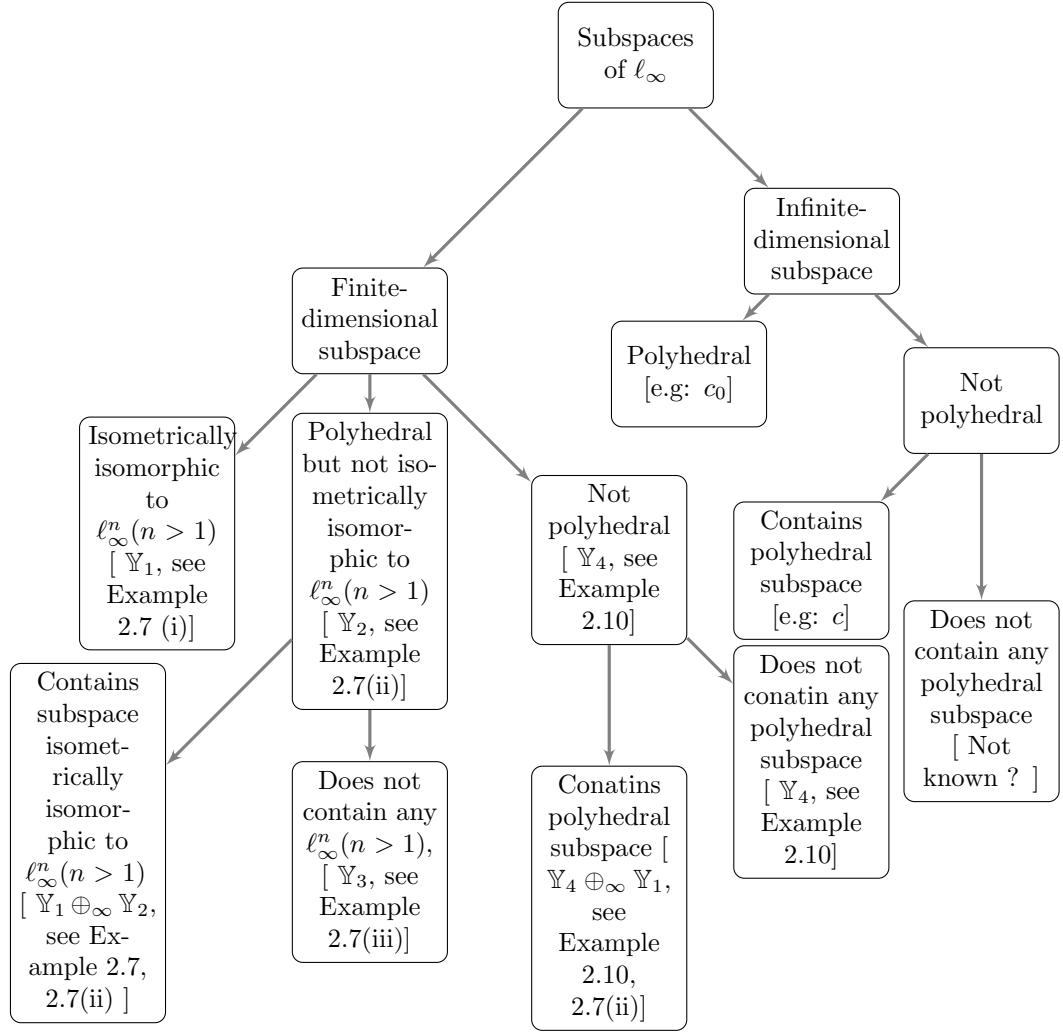
\tikzstyle{decision} = [diamond, draw, fill=blue!20,
text width=2cm, text badly centered, node distance=2cm, inner sep=0pt]
\tikzstyle{block} = [ draw, fill=white!20,text width=5.2em, text centered, rounded corners, minimum height=4em]
\tikzstyle{line} = [draw, very thick, color=black!50, -latex']
\tikzstyle{cloud} = [draw, ellipse,fill=red!20, node distance=.5cm,
minimum height=2em]

\begin{tikzpicture}[scale=1, node distance = 1.5cm, auto]
	\node [block] (block1) {Subspaces of $\ell_{\infty}$};
	\node [block, below right of=block1,  node distance=3.5cm] (block3) {Infinite-dimensional subspace};
	\node [block, below left of=block1,  node distance=5cm] (block4) {Finite-dimensional subspace};
	\node[block, below of=block4, node distance=3.2cm] (block5) {Polyhedral but not isometrically isomorphic to $\ell_{\infty}^n(n>1)$  [ $\mathbb{Y}_2$, see Example \ref{example:example1}(ii)]};
	\node[block, below left of =block4, node distance=4cm] (block6) {Isometrically isomorphic to $\ell_{\infty}^n (n>1)$  [ $\mathbb{Y}_1$, see Example \ref{example:example1} (i)]};
	\node[block, below right of =block4, node distance=4.5cm] (block7) {Not polyhedral [ $\mathbb{Y}_4$, see Example \ref{cos}]};
	
	\node[block, below right of =block7, node distance=3.8cm] (block13) {Does not conatin any polyhedral subspace [ $\mathbb{Y}_4$, see Example \ref{cos}]};
	
	\node[block, below  of =block7, node distance=4.5cm] (block14) { Conatins polyhedral subspace [ $\mathbb{Y}_4 \oplus_{\infty} \mathbb{Y}_1$, see Example \ref{cos}, \ref{example:example1}(ii)]};

	\node[block, below left of= block5, node distance=5.3cm] (block11) {Contains subspace isometrically isomorphic to  $\ell_{\infty}^n (n>1)\,$ [ $\mathbb{Y}_1 \oplus_{\infty} \mathbb{Y}_2,$ see Example \ref{example:example1}, \ref{example:example1}(ii) ]};
	
	\node[block, below  of= block5, node distance=4cm] (block12) { Does not  contain any $\ell_{\infty}^n(n>1),$ [ $\mathbb{Y}_3,$ see Example \ref{example:example1}(iii)]};

	
	\node[block, below left of=block3, node distance=2.5cm] (block9) {Polyhedral [e.g: $c_0$]};
	\node[block,  below right of=block3, node distance=3cm] (block10) {Not polyhedral};
	
	\node[block,  below of=block10, node distance=4cm] (block15) {Does not contain any polyhedral subspace [ Not known ? ] };
	
	\node[block,  below left of=block10, node distance=3.2cm] (block16) {Contains polyhedral subspace [e.g: $c$]};

	\path [line] (block1) -- (block3);
	\path[line] (block1)-- (block4);
	\path[line] (block4)-- (block5);
	\path[line] (block4)-- (block6);
	\path[line] (block4)-- (block7);
	\path[line] (block3)-- (block9);
	\path[line] (block3)-- (block10);
	\path[line](block5)--(block11);
	\path[line](block5)--(block12);
	\path[line](block7)--(block13);
	\path[line](block7)--(block14);
	\path[line](block10)--(block15);
	\path[line](block10)--(block16);
	
	
\end{tikzpicture}

Note that $c$ is the subspace of $\ell_\infty$ consisting of all convergent sequences and $c_0$ is the subspace of $\ell_\infty$ consisting of all sequences converging to $0.$ The following question  remains to be answered :
Does there exist an infinite dimensional subspace $\mathbb{Y}$ of $\ell_\infty$ such that $\mathbb{Y}$  is not polyhedral
and it does not contain any polyhedral subspace? \\

	\section*{Extreme points of the unit ball of a subspace of $\ell_\infty^n$ and extreme contraction}
	
	The extreme points of the unit ball of  $\ell_\infty^n$ are of the form $(\pm1,\pm1, \ldots, \pm1)$ and there are $2^n$ such points.  Given a subspace $\mathbb{W}$ of $ \ell_\infty^n$ we are interested to find the extreme points of the unit ball of the subspace. To do so we introduce the notion of minimal face of an element  in a finite-dimensional polyhedral Banach space. 	Let $\mathbb{X}$ be a finite-dimensional polyhedral Banach space and let $v \in S_{\mathbb{X}}.$ Observe that  if there exists an $i$-face $F$ of $B_{\mathbb{X}}$ such that $v \in F,$ for some $1 \leq i < n-1,$ then there exists an $(i+1)$-face of $B_{\mathbb{X}}$ containing $v.$ If $F$ is an  $i$-face containing $v$  such that there exists no $(i-1)$-face of $B_{\mathbb{X}}$ containing $v,$  then $F$ is said to be the minimal face of $v.$ Note that the minimal face of an element $v \in S_{\mathbb{X}}$ is always unique.  For  two Banach spaces $(\mathbb{Y}_1, \|.\|_{\mathbb{Y}_1})$ and  $(\mathbb{Y}_2, \|.\|_{\mathbb{Y}_2}), $ we use the notation $\mathbb{Y}_1 \oplus_{\infty} \mathbb{Y}_2$ to denote the $\ell_{\infty}$-direct sum of $\mathbb{Y}_1$ and $\mathbb{Y}_2,$i.e., for any $(y_1, y_2) \in \mathbb{Y}_1 \oplus_{\infty} \mathbb{Y}_2,$ $\|(y_1, y_2)\|= \sup \{ \|y_1\|_{\mathbb{Y}_1}, \|y_2\|_{\mathbb{Y}_2}\}.$ 	Further we need the following definitions.
	
	\begin{definition}
		Let $\mathcal{A}= \{ \widetilde{a_1}, \widetilde{a_2}, \ldots, \widetilde{a_m}\}$ be a set of linearly independent elements in $\ell_\infty^n,$ where $\widetilde{a_k}= (a_1^k, a_2^k, \ldots, a_n^k),$ for each $k, 1 \leq k \leq m.$
		
		\begin{itemize}
			\item[(i)] A set $S \subset \{1,2,\ldots,n \}$ is said to be \textit{a $*$-set of $\mathcal{A}$} if there exists $ \widetilde{\beta}=(\beta_1, \beta_2, \ldots, \beta_m) \in \mathbb{R}^m$ such that for any $s \in S,$
			\[	\bigg| \sum_{k=1}^{m} \beta_k a_{s}^k\bigg|= 1 > \max \bigg\{ \bigg| \sum_{k=1}^{m} \beta_k a_j^k\bigg|: j \in \{ 1, 2, \ldots, n\}\setminus S \bigg\}	.	\] 
			In this case $\widetilde{\beta}$ is defined as \textit{a $*$-constant of $S.$}	
			
			\item[(ii)] A $*$-constant $\widetilde{\beta}=(\beta_1, \beta_2, \ldots, \beta_m) \in \mathbb{R}^m$ is said to a \textit{maximal $*$-constant} if there exists no $*$-constant $\widetilde{\alpha}=(\alpha_1, \alpha_2, \ldots, \alpha_m) \in \mathbb{R}^m$ of $ S_1 \subset \{1,2,\ldots,n\}$   such that $ S \subsetneq S_1$ and $ \sum_{k=1}^{m} \beta_k a_s^k= \sum_{k=1}^{m} \alpha_k a_s^k,$ for each $ s \in S.$
		\end{itemize}
	\end{definition}

Next we characterize extreme points of the unit ball of a subspace of a polyhedral Banach space by using the above definition. Before that we need the following proposition.

	\begin{prop}\label{prop:extreme}
	Let $\mathbb{X}$ be a polyhedral Banach space and let $\mathbb{Y}$ be a proper subspace of $\mathbb{X}.$ Let   $F$ be the minimal face of $B_{\mathbb{X}}$ containing $v \in S_{\mathbb{Y}}.$ Then $v$ is an extreme point of $B_{\mathbb{Y}}$ if and only if  $F \cap \mathbb{Y}= \{v\}.$
\end{prop}

\begin{proof}
	Let us first prove the sufficient part of the theorem. Suppose on the contrary that $v \notin Ext(B_{\mathbb{Y}}).$ Then there exist $w_1, w_2 \in S_{\mathbb{Y}}$ such that $v = \frac{1}{2}(w_1+ w_2).$ Since $F$ is a face of $B_{\mathbb{X}}$ such that $v \in F$ and let $F= \delta M \cap S_{\mathbb{X}},$ for some $\delta M,$ a boundary of a closed half space $M$ in $\mathbb{X}$. Now it is easy to observe that there exists a functional $f \in S_{\mathbb{X}^*}$ such that $\delta M = \{ x \in \mathbb{X}: f(x)=1\}.$ Therefore, 
	$	\frac{1}{2}(f(w_1)+ f(w_2))= f(\frac{1}{2}(w_1+ w_2))= f(v)=1 \implies f(w_1)=f(w_2)=1.$ 
	So, $w_1, w_2 \in \delta M \cap S_{\mathbb{Y}}=F.$ That contradicts that $F \cap \mathbb{Y}=\{v\}.$	Next we prove the necessary part.   Let $F$ be the minimal face of $v,$ so $v \in int (F).$ Suppose on the contrary that $F \cap \mathbb{Y} =  \{ v, w\}.$ Let $L$ be the staright line passing through $v, w$ and let $ L$ intersects the boundary of the face $F$ at two points $w_1, w_2 \in F.$ Since $v \in int(F)$ and $v \in L$ it can be easily seen that $v = (1-t) w_1+ tw_2,$ for some $ t \in (0,1).$ This contradicts that $v \in Ext(B_{\mathbb{Y}}).$ This completes the proof of necessary part.
	\end{proof}

	Now we are in a situation to characterize the extreme points of the unit ball of a subspace of $\ell_\infty^n.$ 
	
	\begin{theorem}\label{number:extreme}
		Let  $\mathcal{A}= \{ \widetilde{a_1}, \widetilde{a_2}, \ldots, \widetilde{a_m}\}$ be a set of linearly independent elements in $\ell_\infty^n$ and let $\mathbb{W}= span~ \mathcal{A}.$ Then $\sum_{k=1}^{m} \beta_k \widetilde{a_k}$ is an extreme point of $B_{\mathbb{W}}$ if and only if $\widetilde{\beta}=(\beta_1, \beta_2, \ldots, \beta_m) $ is a maximal $*$-constant.
	\end{theorem}
	
	\begin{proof}
		Let us first prove the sufficient part of the theorem. Let $\widetilde{\beta}$ be a maximal $*$-constant. Then $ \widetilde{\beta}$ is a $*$-constant of an $*$-set $S \subset \{ 1,2, \ldots, n\}.$  This implies \[| \sum_{k=1}^{m} \beta_k a_{s}^k|=1 > \max\bigg\{ | \sum_{k=1}^{m}\beta_k a_q^k|: q \in \{1, 2, \ldots, n\}\setminus S\bigg\},\] for any $s \in S.$ Take the face $F$ of $B_{\ell_\infty^n}$ such that for any $ \widetilde{x}=(x_1, x_2, \ldots, x_n) \in F,$ we have $x_s= \sum_{k=1}^{m} \beta_k a_s^k,$ for any $s \in S.$ Clearly,   $ \sum_{k=1}^{m} \beta_k \widetilde{a_k} \in F \cap \mathbb{W}$ and $F$ is the minimal face of $\sum_{k=1}^m \beta_k \widetilde{a_k}.$  Following Proposition \ref{prop:extreme} we only need to verify that $F \cap \mathbb{W}= \{ \sum_{k=1}\beta_k \widetilde{a_k}\}$ to conclude that $ \sum_{k=1}\beta_k \widetilde{a_k}$  is an extreme point of $B_{\mathbb{W}}.$  If not, then suppose that $| F \cap \mathbb{W}|> 1.$ Proceeding in the same way as  in the necessary part of Proposition \ref{prop:extreme}, we observe that there exists a face $F_1 \subsetneq F$ of $B_{\ell_\infty^n}$ such that $ F_1 \cap \mathbb{W} \neq \emptyset.$ Let $ \sum_{k=1}^{m} \gamma_k \widetilde{a_k} \in F_1 \cap \mathbb{W}.$ Since $ F_1 \subsetneq F,$ it is immediate that $ \sum_{k=1}^{m} \gamma_k a_s^k= \sum_{k=1}^{m} \beta_k a_s^k,$ for any $s \in S$ and $ | \sum_{k=1}^{m} \gamma_k a_p^k| =1,$ for some $p \notin S.$ Take $ S_1 = S \cup \{ p\}.$ Clearly, $ \widetilde{\gamma}=(\gamma_1, \gamma_2, \ldots, \gamma_m)$ is a $*$-constant of $S_1,$ this contradicts that $ \widetilde{\beta}$ is a maximal $*$-constant. This completes the proof of the sufficient part. 	We next prove the necessary part. Suppose that $\sum_{k=1}^{m} \beta_k \widetilde{a_k}$ is an extreme point of $B_{\mathbb{W}}$ and $ F$ is the minimal face of $ \sum_{k=1}^{m} \beta_k \widetilde{a_k}.$ Then from Proposition \ref{prop:extreme}, we get  $ F \cap \mathbb{W}= \{ \sum_{k=1}^{m} \beta_k \widetilde{a_k}\}.$ Without loss of generality we assume that $F$ is an $(n-p)$-face of $B_{\ell_\infty^n}$ and for any $ \widetilde{x}=(x_1, x_2, \ldots, x_n) \in F,$ we have $| x_{i_t}| =1,$ for any $ t \in \{ 1, 2, \ldots, p\}.$ Note that $\{ i_1, i_2, \ldots, i_p\}  \subset \{1,2, \ldots, n\}.$ It is now easy to observe that $ \widetilde{\beta}$ is a $*$-constant of $S= \{ i_1, i_2, \ldots, i_p\}.$ Suppose on the contrary that $ \widetilde{\beta}$ is not a maximal $*$-constant. Then there exists a $*$-set $S_1$ such that $ S \subsetneq S_1$ and for any $1 \leq t \leq p,$ $ \sum_{k=1}^{m} \beta_k a_{i_t}^k= \sum_{k=1}^{m} \alpha_k a_{i_t}^k,$ for some $*$-constant $ \widetilde{\alpha}=(\alpha_1, \alpha_2, \ldots, \alpha_m)\in \mathbb{R}^m$ of $S_1.$ It is easy to observe that $ \sum_{k=1}^{m} \alpha_k \widetilde{a_k} \in F \cap \mathbb{W}.$ This contradicts that $ F \cap \mathbb{W}= \{ \sum_{k=1}^{m} \beta_k \widetilde{a_k}\}.$  This completes the proof.
	\end{proof}

\begin{remark}
	Let $ \mathbb{Y}_1, \mathbb{Y}_2$ be two finite-dimensional subspace of $\ell_\infty.$ Then a necessary condition for having an isometric isomorphism between $ \mathbb{Y}_1$ and $ \mathbb{Y}_2$ can be obtained from Theorem \ref{number:facet} and Theorem \ref{number:extreme}. 
\end{remark}

In the rest of the article we focus on  the number of extreme contraction of the space $\mathbb{L}(\mathbb{X}, \ell_{\infty}^n),$ where $\mathbb{X}$ is a finite-dimensional polyhedral Banach space. To do so we need the following theorem.

\begin{theorem}\label{directsum}
		Let $\mathbb{X}$ be an $m$-dimensional polyhedral Banach space  and let $\mathbb{Y}$ be an $n$-dimensional Banach space. Suppose that $|Ext(B_{\mathbb{X}})|=2r.$ Then $\mathbb{L}(\mathbb{X}, \mathbb{Y})$ is isometrically isomorphic  to an $mn$-dimensional subspace of 	$ \underbrace{\mathbb{Y} \oplus_{\infty} \mathbb{Y}\oplus_{\infty} \ldots \oplus_{\infty} \mathbb{Y}}_{r-times}.$
\end{theorem}

\begin{proof}
		Let $Ext(B_{\mathbb{X}})= \{ \pm v_1, \pm v_2, \ldots, \pm v_r\}.$ We construct a linear transformation $f$ from $\mathbb{L}(\mathbb{X}, \mathbb{Y}) $ to  $\underbrace{\mathbb{Y} \oplus_{\infty} \mathbb{Y}\oplus_{\infty} \ldots \oplus_{\infty} \mathbb{Y}}_{r-times}$ as follows:		
			\begin{eqnarray*}
										&f(A)& = (Av_1, Av_2, \ldots, Av_r),
		\end{eqnarray*}
		for each $A \in \mathbb{L}(\mathbb{X}, \mathbb{Y}).$  Clearly $f$ is well-defined and $f$ is linear. 	Next,	for each $A \in \mathbb{L}(\mathbb{X}, \mathbb{Y}),$ we have 
		\begin{eqnarray*}
			\|A\|&=& \sup\{ \|Ax\|: x \in S_{\mathbb{X}}\}\\
			&=& \max \{ \|Ax\|: x \in Ext(B_{\mathbb{X}})\}\\
			&=& \max \{ \|Av_1\|, \|Av_2\|, \ldots, \|Av_r\|\}\\
			&=& \| (Av_1, Av_2, \ldots, Av_r)\|_{\infty}\\
			&=& \|f(A)\|,
		\end{eqnarray*}
		so that $f$ is an isometry.  Since 	$ dim~\mathbb{L}(\mathbb{X}, \mathbb{Y}) = \underbrace{\mathbb{Y} \oplus_{\infty} \mathbb{Y}\oplus_{\infty} \ldots \oplus_{\infty} \mathbb{Y}}_{r-times} = mn,$ it follows that $f$ is an isometric isomorphism.
\end{proof}
Next we provide an easy proposition in which we characterize extreme points of a finite-dimensional Banach space $\mathbb{Y}$ when it is the $\ell_\infty$-direct sum  decomposition of finitely many spaces.

\begin{prop}\label{direct}
	Let $ \mathbb{Y} = \mathbb{Y}_1 \oplus_{\infty} \mathbb{Y}_2 \oplus_{\infty} \ldots \oplus_{\infty} \mathbb{Y}_m , $ where $\mathbb{Y}_1, \mathbb{Y}_2, \ldots, \mathbb{Y}_m$ are finite-dimensional polyhedral Banach spaces. Then $(y_1, y_2, \ldots, y_m) \in Ext(B_{ \mathbb{Y}}) $ if and only if $y_i \in Ext(B_{\mathbb{Y}_i}),$ for each $ i, 1 \leq i \leq m$ and consequently, 
	\[ |Ext( B_{ \mathbb{Y}})|= |Ext(B_{\mathbb{Y}_1})| \times | Ext(B_{\mathbb{Y}_2})| \times \ldots \times |Ext(B_{\mathbb{Y}_m})|.\] 
\end{prop}

\begin{proof}
 Let	$(y_1, y_2, \ldots, y_m) \in Ext(B_{\mathbb{Y}}).$ Suppose on the contrary that $y_i \notin Ext(B_{\mathbb{Y}_i}),$ for some $i \in \{ 1, 2, \ldots, m\}.$ Then there exists $x_i, z_i \in S_{\mathbb{Y}_i}$ such that $y_i = \frac{1}{2} (x_i+ z_i).$ So, $(y_1, y_2, \ldots, y_m) = \frac{1}{2} [( y_1, \ldots, y_{i-1}, x_i, y_{i+1}, \ldots, y_m)+ ( y_1, \ldots, y_{i-1}, z_i, y_{i+1}, \ldots, y_m)].$ That contradicts that $(y_1, y_2, \ldots, y_m) \in Ext(B_{\mathbb{Y}}).$ This completes the proof of necessary part. 
 Conversely, let $y_i \in Ext(B_{\mathbb{Y}_i})$  for each $i, 1 \leq i \leq n.$  Let $(y_1, y_2, \ldots, y_m)= \frac{1}{2} [(x_1, x_2, \ldots, x_m)+(z_1, z_2, \ldots, z_m)] ,$ where $x_i, z_i \in S_{\mathbb{Y}_i}.$ So, $y_i = \frac{1}{2} (x_i+ z_i),$ for any $i, 1 \leq i \leq n.$ Since $y_i \in 
 Ext(B_{\mathbb{Y}_i}),$ we have $x_i = y_i= z_i,$ for any $1 \leq i \leq m. $ So, $(y_1, y_2, \ldots, y_m) \in Ext(B_{\mathbb{Y}}).$ This completes the proof.\\
\end{proof}

\begin{remark}\label{remark}
	Let $\mathbb{X}$ be an $m$-dimensional polyhedral Banach space such that $Ext(B_{\mathbb{X}})= \{ \pm \widetilde{v_1}, \pm \widetilde{v_2}, \ldots, \pm \widetilde{v_r}\},$ where $ \widetilde{v_k}= (v_1^k, v_2^k, \ldots, v_m^k).$   Then $\mathbb{L}(\mathbb{X}, \ell_\infty^n)$ is isometrically isomorphic  to an  $mn$-dimensional subspace $\mathbb{Y}$ of $\ell_\infty^{rn}.$ Moreover we observe that $ \mathbb{Y}=  \underbrace{\mathbb{W} \oplus_{\infty} \mathbb{W}\oplus_{\infty} \ldots \oplus_{\infty} \mathbb{W}}_{n-times},$ where $\mathbb{W}= span \{ \widetilde{w_1}, \widetilde{w_2}, \ldots, \widetilde{w_m} \},$ where $\widetilde{w_k}= ( v_k^1, v_k^2, \ldots, v_k^r) \in \ell_\infty^r.$ Therefore, $|Ext(B_{\mathbb{L}(\mathbb{X}, \ell_\infty^n)})|= |Ext(B_{\mathbb{W}})|^n.$\\
	\end{remark}
\begin{remark}
	Using Theorem \ref{directsum} and Proposition \ref{direct},  we observe that
	\begin{itemize} 
		\item[(i)]  $\mathbb{L}(\ell_1^m, \ell_\infty^n)$ is  isometrically isomorphic to $\ell_\infty^{mn}.$
		\item[(ii)]  If $\mathbb{Y}$ is a finite-dimensional Banach space, then $\mathbb{L}(\ell_1^m, \mathbb{Y})$ is  isometrically isomorphic to $ \underbrace{\mathbb{Y} \oplus_{\infty} \mathbb{Y}\oplus_{\infty} \ldots \oplus_{\infty} \mathbb{Y}}_{m-times}.$ Moreover, if $f $ is an isometric isomorphism between $ \underbrace{\mathbb{Y} \oplus_{\infty} \mathbb{Y}\oplus_{\infty} \ldots \oplus_{\infty} \mathbb{Y}}_{m-times}$ and  $\mathbb{L}(\ell_1^m, \mathbb{Y}),$ then
		$f(y_1, y_2, \ldots, y_r) \in Ext(B_{\mathbb{L}(\ell_1^m, \mathbb{Y})}) $ if and only if $ y_i \in Ext(B_{\mathbb{X}}),$ for each $i,  1 \leq i \leq r$ and $|Ext(B_{\mathbb{L}(\ell_1^m, \mathbb{Y})})|= |Ext(B_{\mathbb{Y}})|^n,$ whenever $\mathbb{Y}$ is polyhedral.
		\item[(iii)] $(\mathbb{L}(\ell_1^m, \mathbb{Y}))^*$ is an isometrically isomorphic to $ \underbrace{\mathbb{Y}^* \oplus_{1} \mathbb{Y}^* \oplus_{1} \ldots \oplus_{1} \mathbb{Y}^*}_{m-times}.$\\
	\end{itemize}
	\end{remark}

In the following lemma we explore the $*$-Property of direct sum of finitely many subspaces of $\ell_{\infty}^n.$ If $\mathcal{A}= \{ \widetilde{a_1}, \widetilde{a_2}, \ldots, \widetilde{a_m}\}$ be a set of linearly independent elements in $\ell_\infty^n$ then the set $\{ (\widetilde{a_1}, \theta), (\widetilde{a_2}, \theta), \ldots, (\widetilde{a_m}, \theta), $ $ (\theta, \widetilde{a_1}),   (\theta, \widetilde{a_2} ), \ldots, (\theta, \widetilde{a_m})\} $ is denoted by $\mathcal{A} \oplus \mathcal{A}.$

\begin{lemma}\label{lemma}
	Let  $\mathcal{A}= \{ \widetilde{a_1}, \widetilde{a_2}, \ldots, \widetilde{a_m}\}$ be a set of linearly independent elements in $\ell_\infty^n$ and let $\mathbb{W}= span~ \mathcal{A}.$
	Let $\mathbb{Y} =\mathbb{W} \oplus_{\infty} \mathbb{W}. $ If the number of nonequivalent components of  $\mathcal{A}$ satisfying  the $*$-Property is $r$ then the number of nonequivalent components of $\mathcal{B}$ that  satisfies the $*$-Property is $2r,$ where $\mathcal{B}= \mathcal{A} \oplus \mathcal{A}.$
\end{lemma}

\begin{proof}
	  Suppose that $\widetilde{a_k}= (a_1^k, a_2^k, \ldots, a_n^k),$ for each $k, 1 \leq k \leq m.$ Clearly, $\mathcal{B}$ is a basis of $\mathbb{Y}.$  Suppose that the $i$-th component of $\mathcal{A}$ satisfies the $*$-Property, then there exists $(\beta_1, \beta_2, \ldots, \beta_m) \in \mathbb{R}^m$ satisfying 
     	\[
     | \sum_{k=1}^{m} \beta_k a_i^k | > \max \bigg\{	| \sum_{k=1}^{m} \beta_k a_j^k |: j \in \{1, 2, \ldots, n\} \setminus E_{i}\bigg\}.
     \]
     Observe that the $i$-th component and the $(m+i)$-th component of $\mathcal{B}$ is $( a_i^1, a_i^2, \ldots, a_i^m,$ $ 0, 0, \ldots, 0) \in \mathbb{R}^m$ and $( 0, 0, \ldots, 0, a_i^1, a_i^2, \ldots, a_i^m) \in \mathbb{R}^m,$ respectively. Then by choosing the scalars $(\beta_1, \beta_2, \ldots, \beta_m, 0, 0, \ldots, 0) \in \mathbb{R}^m$ and $ ( 0, 0, \ldots, 0, \beta_1, \beta_2, \ldots, \beta_m) $ $ \in \mathbb{R}^m,$ it is easy to verify that the $i$-th component and the $(m+i)$-th component satisfy the $*$-Property, respectively. Therefore, if the number of nonequivalent components of $\mathcal{A}$ satisfying the $*$-Property is $r,$ then the number of nonequivalent components of $\mathcal{B}$ satisfying the $*$-Property is $2r.$ Hence the lemma.
 \end{proof}

	Proceeding similarly  the above result can also be proved for the case $$\mathbb{Y}= \underbrace{\mathbb{W} \oplus_{\infty} \mathbb{W}\oplus_{\infty} \ldots \oplus_{\infty} \mathbb{W}}_{n-times}.$$

In the following theorem we find the number of extreme contractions in $\mathbb{L}(\mathbb{X}, \ell_{\infty}^n),$ where $S_\mathbb{X}$ is a $2$-dimensional regular polygon, which generalizes \cite[Th. 2.7]{RRBS}.

\begin{theorem}\label{regular}
	Let $\mathbb{X}$ be a $2$-dimensional  polyhedral Banach space such that the unit sphere of $\mathbb{X}$ is a regular $2r$-sided polygon. Then
the number of facets of the unit ball of  $\mathbb{L}(\mathbb{X}, \ell_\infty^n)$ is $2rn$ and 	$|Ext(B_{\mathbb{L}(\mathbb{X}, \ell_\infty^n)})|= (2r)^n.$
	
\end{theorem}

\begin{proof}
	Without loss of generality we assume that 
	$Ext(B_{\mathbb{X}})= \bigg\{ \pm (1,0), \pm \bigg(\cos\frac{\pi}{r}, \sin \frac{\pi}{r}\bigg), $\\ $  \pm \bigg(\cos\frac{2\pi}{r}, \sin \frac{2\pi}{r} \bigg), \ldots, \pm \bigg(\cos\frac{(r-1)\pi}{r}, \sin \frac{(r-1)\pi}{r} \bigg) \bigg\}.$
	Suppose that 
	\[
       \mathcal{A}= 	\bigg\{ \bigg( 1,\cos\frac{\pi}{r}, \cos\frac{2\pi}{r}, \ldots, \cos\frac{(r-1)\pi}{r} \bigg), \bigg( 0,  \sin \frac{\pi}{r}, \sin \frac{2\pi}{r}, \ldots, \sin \frac{(r-1)\pi}{r}\bigg)   \bigg\} \subset \ell_{\infty}^r.
	\]
	Let $\mathbb{W}= span ~ \mathcal{A}.$ Let $\mathbb{V}=  \underbrace{\mathbb{W} \oplus_{\infty} \mathbb{W}\oplus_{\infty} \ldots \oplus_{\infty} \mathbb{W}}_{n-times}$ and let $\mathcal{B}= \underbrace{\mathcal{A} \oplus \mathcal{A} \oplus \ldots \oplus \mathcal{A}}_{n-times}.$ Clearly, $\mathcal{B}$ is  a basis of $\mathbb{V}.$
	Then  $\mathbb{L}(\mathbb{X}, \ell_\infty^n)$ (see Remark \ref{remark}) is isometrically isomorphic to $ \mathbb{V},$ which is a $2n$-dimensional subspace of $\ell_\infty^{rn}.$ For each $i, 1 \leq i \leq r,$ the $i$-th component of $\mathcal{A}$ is $\bigg( \cos\frac{(i-1)\pi}{r}, \sin \frac{(i-1) \pi}{r}\bigg)$ and each of them   satisfies the $*$-Property.
From Lemma \ref{lemma}, each component of $ \mathcal{B} $ satisfies the $*$-Property. Therefore, from Theorem \ref{star}, the number of facets of  $ B_{\mathbb{V}} $ is $2rn.$ Since  $\mathbb{L}(\mathbb{X}, \ell_\infty^n)$  is isometrically isomorphic to $ \mathbb{V},$ so the number of facets of the unit ball of  $\mathbb{L}(\mathbb{X}, \ell_{\infty}^n)$ is $2rn.$ 
	
	Since	$ dim(\mathbb{W})=2,$  $ |Ext(B_{\mathbb{W}})|= \textit{ number of facets of $B_{\mathbb{W}}$} =  2r.$ From Proposition \ref{direct}, $|Ext(B_{\mathbb{L}(\mathbb{X}, \ell_{\infty}^n)})|=|Ext(B_\mathbb{V})|= |Ext(B_{\mathbb{W}}) |^n= (2r)^n.$ Hence the theorem.\\
\end{proof}

In \cite[Th. 2.7]{RRBS}, it is proved that if the unit sphere of  $\mathbb{X}$ is a regular hexagon and $\mathbb{Y}= \ell_{\infty}^2,$ then the total number of extreme contractions in $\mathbb{L}(\mathbb{X}, \mathbb{Y})$ is $36.$ Putting $n=2$ and $r=3$ in Theorem \ref{regular}(ii) we get the same result. \\

We end this article with following example in which we compute the number of extreme contractions in $\mathbb{L}(\mathbb{X}, \ell_{\infty}^n)$ with the help of our newly introduced maximal $*$-constant.


\begin{example}\label{last-ex}
	Let $\mathbb{X}$ be a $3$-dimensional polyhedral Banach space with  $Ext(B_{\mathbb{X}})= \{ \pm (1, 0, 0), \pm (0, 1, 0), \pm (1, 1, 0), \pm (0, 0, 1)\}.$ 	Let $ \mathcal{A}= \{ (1,0,1,0), (0, 1, 1, 0),  (0, 0, 0, 1)\} $ and let  $ \mathbb{W}= span ~ \mathcal{A}.$ Let $ \mathbb{Y} = \underbrace{\mathbb{W} \oplus_{\infty} \mathbb{W}\oplus_{\infty} \ldots \oplus_{\infty} \mathbb{W}}_{n-times} $ and let $\mathcal{B} = \underbrace{\mathcal{A} \oplus \mathcal{A} \oplus \ldots \oplus \mathcal{A}}_{n-times}.$ 	We observe that $\mathbb{L}(\mathbb{X}, \ell_\infty^n)$ is isometrically isomorphic to $\mathbb{Y}.$   Now  each component of $\mathcal{A}$ satisfies the $*$-Property and  from Lemma \ref{lemma} it follows that  each component of $\mathcal{B}$  satisfies the $*$-Property.   Since the number of components of $\mathcal{B}$ is $4n,$ then from Theorem \ref{star}, the number of facets of $B_{\mathbb{Y}}=8n.$
		It is straightforward to verify that $ \pm (1, -1, 1), \pm (1, -1, -1), \pm(0, 1, 1), \pm (0, 1, -1), $ $ \pm (1, 0, 1), \pm (1, 0, -1)$ are the maximal  $*$-constant of $\mathbb{W}.$ Therefore, 
	$Ext(B_{\mathbb{W}})= \{ \pm (1, -1, 0, 1), \pm (1, -1, 0, -1),$ $ \pm (0, 1, 1, 1), \pm ( 0, 1, 1, -1), \pm (1, 0, 1, 1), \pm (1, 0, 1, -1)\}.$
	So, $| Ext(B_\mathbb{Y})|= (12)^n.$ Therefore, $|Ext(\mathbb{L}(\mathbb{X}, \ell_{\infty}^n))|=(12)^n.$
	\end{example}

Using the maximal $*$-constant and following  similar technique as in  Example \ref{last-ex}, we can easily determine the number of extreme contraction of $\mathbb{L}(\mathbb{X}, \ell_{\infty}^n),$ whenever the extreme points of $B_{\mathbb{X}}$ are  given.

\end{document}